\documentclass[oneside]{article}
\usepackage{amsmath,amssymb,amscd,amsthm,verbatim,alltt,amsfonts,array}
\usepackage{mathrsfs}
\usepackage[english]{babel}
\usepackage{latexsym}
\usepackage{amssymb}
\usepackage{euscript}
\usepackage{graphicx}
\usepackage{arydshln}

\newtheorem{theorem}{Theorem}
\theoremstyle{plain}

\newtheorem{corollary}{Corollary}

\newtheorem{lemma}{Lemma}

\numberwithin{equation}{section}

\begin{document}

{\footnotesize%
\hfill
}

  \vskip 1.2 true cm

\begin{center} {\bf A remark on extremal stability of configuration spaces} \\
          {by}\\
{\sc Muhammad Yameen}
\end{center}

\pagestyle{myheadings}
\markboth{A remark extremal stability of configuration spaces}{Muhammad Yameen}

\begin{abstract}
We explicitly study the extremal stability of configuration spaces of complex projective spaces of any dimension, and show that the homology groups are vanish in extremal stable range. As a consequence, we give an affirmative answer of the question of Knudsen, Miller and Tosteson.
\end{abstract}

\begin{quotation}

\end{quotation}

\thispagestyle{empty}

\section{Introduction}\label{Sec1}

\label{sec:intro}
For any manifold $M$, let
$$F_{k}(M):=\{(x_{1},\ldots,x_{k})\in M^{k}| x_{i}\neq x_{j}\,for\,i\neq j\}$$
be the configuration space of $k$ distinct ordered points in $M$ with induced topology. The symmetric group $S_{k}$ acts on $F_{k}(M)$ by permuting the coordinates. The quotient $$C_{k}(M):=F_{k}(M)/S_{k} $$
is the unordered configuration space with quotient topology. It is a
classical problem in algebraic topology to
understand the homology and cohomology of such spaces. Arnold proved integral cohomological stability for $\mathbb{R}^{2}$:
$$H^{i}(C_{2i-2}(\mathbb{R}^{2});\mathbb{Z})\cong H^{i}(C_{2i-1}(\mathbb{R}^{2});\mathbb{Z})  \cong H^{i}(C_{2i}(\mathbb{R}^{2});\mathbb{Z})\cong\ldots$$
The isomorphisms (for $k$ large depending on $i$)
$$H^{i}(C_{k}(M);\mathbb{Z})\cong H^{i}(C_{k+1}(M);\mathbb{Z})\cong H^{i}(C_{k+2}(M);\mathbb{Z})\cong\ldots$$
were generalized for open manifolds by McDuff \cite{MD} and Segal \cite{S}. Using representation stability, Church \cite{C} proved that
$$H^{i}(C_{k}(M);\mathbb{Q})\cong H^{i}(C_{k+1}(M);\mathbb{Q})\cong H^{i}(C_{k+2}(M);\mathbb{Q})\cong\ldots$$
for $k>i$ and $M$ a connected oriented manifold of finite type. This result was extended by Kundsen \cite{Kn}. More recently, Knudsen, Miller and Tosteson \cite{KMT} study the extremal stability (the stability of top cohomology) of unordered configuration spaces of manifolds. They asked the following question:\\\\
\textbf{Question.} (see Question 4.10 of \cite{KMT}) Suppose that $H_{d-1}(M;\mathbb{Q})=0.$ For $i\in\mathbb{N},$ is the Hilbert function 
$$k\mapsto \text{dim}H_{k(d-2)+i}(C_{k}(M);\mathbb{Q})$$
eventually a quasi-polynomial?\\\\
Here, the dimension of $M$ is $d.$ The above question has an affirmative answer under the further assumption that $H_{d-2}(M;Q) = 0$ by Theorem 4.9 of \cite{KMT}. Here is our main result:
\begin{theorem}\label{theorem-main}
For $i, m\in\mathbb{N},$ the Hilbert function 
$$k\mapsto \emph{dim}H_{k(2m-2)+i}(C_{k}(\emph{CP}^{m});\mathbb{Q})$$
is eventually a quasi-polynomial.
\end{theorem}

\section{Chevalley–Eilenberg complex}
The cohomology and homology of configuration spaces have received a lot of attention, since the work by Arnold \cite{A}. B\"{o}digheimer--Cohen--Taylor \cite{B-C-T} studied the homology of $C_{k}(M)$ in the case of odd dimensional manifolds. F\'{e}lix--Thomas \cite{F-Th} (see also \cite{F-Ta}) constructed a Sullivan model for the rational cohomology of configuration spaces of closed oriented even dimensional manifolds. Furthermore, Knudsen \cite{Kn} extended the result of F\'{e}lix--Thomas for general even dimensional manifolds.

Let $dim(M)=2m.$ Throughout the paper, we will consider the homology and cohomology over $\mathbb{Q}$. The symmetric algebra $Sym(A^{*})$ is the tensor product of a polynomial algebra and an exterior algebra:
$$ Sym(A^{*})=\bigoplus_{k\geq0}Sym^{k}(A^{*})=Poly(A^{even})\bigotimes Ext(A^{odd}), $$
where $Sym^{k}$ is generated by the monomials of length $k.$ The $n$-th suspension of the graded vector space $V$ is the graded vector space $V[n]$ with
$V[n]_{i} = V_{i-n},$ and the element of $V[n]$ corresponding to $a\in V$ is denoted $s^{n}a$. We write $H_{-*}(M;\mathbb{Q})$ for the graded vector space whose degree $-i$ part is
the $i$-th homology group of $M.$

Consider two graded vector spaces $V^{*}=H_{-*}(M;\mathbb{Q})[2m],$  $W^{*}=H_{-*}(M;\mathbb{Q})[4m-1]$
$$ V^{*}=\bigoplus_{i=0}^{2m}V^{i},\,\,W^{*}=\bigoplus_{j=2m-1}^{4m-1}W^{j},$$
and a degree 1 linear map $\partial:$
$$\partial|_{V^{*}}=0,\quad \partial|_{W^{*}}:\,W^{*}\simeq H_{*}(M;\mathbb{Q})\xrightarrow{\Delta} Sym^{2}(V^{*})\simeq  Sym^{2}(H_{*}(M;\mathbb{Q})),$$
where
$\Delta$ is diagonal co-multiplication corresponding to cup product. We choose bases in $V^{i}$ and $W^{j}$ as $$ V^{i}=\mathbb{Q}\langle v_{i,1},v_{i,2},\ldots\rangle,\quad W^{j}=\mathbb{Q}\langle w_{j,1},w_{j,2},\ldots\rangle $$
(the degree of an element is marked by the first lower index; $x_{i}^{q}$ stands for the product $x_{i}\wedge x_{i}\wedge\ldots\wedge x_{i}$ of $q$-factors). Always we take $V^{0}=\mathbb{Q}\langle v_{0}\rangle$. Now we consider the graded algebra:

$$ \Omega^{*,*}_{k}(M)=\bigoplus_{i\geq 0}\bigoplus_{\omega=0}^{\left\lfloor\frac{k}{2}\right\rfloor}
\Omega^{i,\omega}_{k}(M)=\bigoplus_{\omega=0}^{\left\lfloor\frac{k}{2}\right\rfloor}\,(Sym^{k-2\omega}(V^{*})\otimes Sym^{\omega}(W^{*})) $$
the (total) degree $i$ is given by the grading of $V^{*}$ and $W^{*},$ where $\omega$ is a weight grading. The differential $\partial$ extends over graded algebra by using Leibniz's rule. By definition of differential, we have 
$$\partial:\Omega^{*,*}_{k}(M)\longrightarrow\Omega^{*+1,*-1}_{k}(M).$$

\begin{theorem}\label{th.2} (\cite{F-Th} \cite{F-Ta} \cite{Kn}) 
	For a connected closed oriented manifold $M$ of even dimension, we have $$H^{*}(C_{k}(M))\simeq H^{*}(\Omega^{*,*}_{k}(M),\partial).$$
\end{theorem}
The complex $(\Omega^{*,*}_{k}(M),\partial)$ is Chevalley–Eilenberg type complex. The Chevalley–Eilenberg complex has been a ubiquitous presence in the study of the configuration spaces of manifolds. Prominent examples of its appearance include the work of B\"{o}digheimer--Cohen--Taylor \cite{B-C-T} and F\'{e}lix--Thomas \cite{F-Th}, building on McDuff’s foundational work \cite{MD}; the
work of F\'{e}lix--Tanr\'{e} \cite{F-Ta} following Totaro \cite{T}; and the work of Knudsen \cite{Kn} building on work of Ayala--Francis \cite{AF}.
\section{Proof of main Theorem}
The cohomology ring of $\text{CP}^{m}$ is the truncated polynomial ring with single generator:
	$$H^{*}(\text{CP}^{m};\mathbb{Q})=\frac{\mathbb{Q}[x]}{\langle x^{m}\rangle},\quad \mbox{where  } deg(x)=2.$$ The corresponding two graded vector spaces are
	$$V^{*}=\langle v_{0},v_{2},\ldots,v_{2m}\rangle, \quad W^{*}=\langle w_{2m-3},w_{2m+3},\ldots,w_{4m-3}\rangle.$$
The differential $\partial$ is define on $V^{*}$ and $W^{*}$ as
\begin{align*}
\partial(v_{2i})=&\,\,0\qquad\qquad\qquad\quad 0\leq i\leq m,\\
\partial(W_{2i-1})=&\sum_{\substack{a+b=i,\\0\leq a, b\leq m}}v_{2a}v_{2b}\quad m\leq i\leq 2m.
\end{align*}
\begin{lemma}\label{lemma1}
The subspace $\Omega_{k-2}^{*,*}(\emph{CP}^{m}).(v_{2m}^{2}, w_{4m-1})< \Omega_{k}^{*,*}(\emph{CP}^{m})$ is acyclic for $k\geq 2.$
\end{lemma}
\begin{proof}
An element in $\Omega_{k-2}^{*,*}(\text{CP}^{m}).(v_{2m}^{2}, w_{4m-1})$ has a unique expansion $v_{2m}^{2}A+Bw_{4m-1},$ where $A$ and $B$ have no monomial containing $w_{4m-1}.$ The operator $$h(v_{2m}^{2}A+Bw_{4m-1})=w_{4m-1}A$$ gives a homotopy $id\simeq 0.$

\end{proof}
We denote the reduced complex $(\Omega_{k}^{*,*}(\text{CP}^{m})/\Omega_{k-2}^{*,*}(\text{CP}^{m}).(v_{2m}^{2}, w_{4m-1}),\partial_{\text{induced}})$ by 
$$({}^{r}\Omega_{k}^{*,*}(\text{CP}^{m}),\partial).$$
\begin{corollary}\label{corollary1}
For $k\geq2,$ we have isomorphism $H^{*}({}^{r}\Omega_{k}^{*,*}(\emph{CP}^{m}),\partial)\cong H^{*}(C_{k}(\emph{CP}^{m})).$
\end{corollary}

\begin{lemma}\label{lemma2}
The cohomology groups $H^{k(2m-2)+i}(C_{k}(\emph{CP}^{m}))$ are eventually vanish for $m\geq1$ and $i\geq4.$
\end{lemma}
\begin{proof}
For $k\geq 4,$ the highest degree element in the reduce complex $({}^{r}\Omega_{k}^{*,*}(\text{CP}^{m}),\partial)$ is $v_{2m-2}^{k-3}v_{2m}w_{4m-3}.$ The degree of $v_{2m-2}^{k-3}v_{2m}w_{4m-3}$ is $(2m-2)k+3.$ This implies that the cohomology groups $H^{k(2m-2)+i}(C_{k}(\text{CP}^{m}))$ are vanish for $m\geq1,$ $i\geq4$ and $k\geq 4.$ 
\end{proof}
\textit{Proof of Theorem \ref{theorem-main}.} If $m=1,$ then we have $H^{k(2m-2)+i}(C_{k}(\text{CP}^{m}))=H^{i}(C_{k}(\text{CP}^{1})).$ The church homological stability of configuration spaces (see Corollary 3 of \cite{C}) implies that the function 
$$k\mapsto \text{dim}H^{i}(C_{k}(\text{CP}^{1}))$$
is eventually constant for each $i\in \mathbb{N}.$

Let $m\geq 2$ and $k\geq 8.$ From Lemma \ref{lemma2}, we have vanishing:
$$H^{k(2m-2)+i}(C_{k}(\text{CP}^{m}))=0\quad \text{for } i\geq 4.$$ Now, we just focus on the cohomology groups $H^{k(2m-2)+i}(C_{k}(\text{CP}^{m}))$ for $i=1, 2, 3.$ There is no element of degree higher than $k(2m-2)+3$ in reduced complex $({}^{r}\Omega_{k}^{*,*}(\text{CP}^{m}),\partial).$ The elements of degrees $k(2m-2),$ $k(2m-2)+1,$ $k(2m-2)+2$ and $k(2m-2)+3$ are concentrated in the following two sub-complex:
$$0\longrightarrow {}^{r}\Omega_{k}^{k(2m-2)+2,2}(\text{CP}^{m})\longrightarrow {}^{r}\Omega_{k}^{k(2m-2)+3,1}(\text{CP}^{m})\longrightarrow 0$$
$$\ldots\longrightarrow{}^{r}\Omega_{k}^{k(2m-2),2}(\text{CP}^{m})\longrightarrow {}^{r}\Omega_{k}^{k(2m-2)+1,1}(\text{CP}^{m})\longrightarrow {}^{r}\Omega_{k}^{k(2m-2)+2,0}(\text{CP}^{m})\longrightarrow 0$$
where 
\begin{align*}
{}^{r}\Omega_{k}^{k(2m-2)+3,1}=&\langle v_{2m-2}^{k-3}v_{2m}w_{4m-3}\rangle,\\
{}^{r}\Omega_{k}^{k(2m-2)+2,2}=&\langle v_{2m-2}^{k-5}v_{2m}w_{4m-5}w_{4m-3}\rangle,\\
{}^{r}\Omega_{k}^{k(2m-2)+2,0}=&\langle v_{2m-2}^{k-1}v_{2m}\rangle,\\
{}^{r}\Omega_{k}^{k(2m-2)+1,1}=&\langle v_{2m-4}v_{2m-2}^{k-4}v_{2m}w_{4m-3}, v_{2m-2}^{k-3}v_{2m}w_{4m-5}, v_{2m-2}^{k-2}w_{4m-3}\rangle,\\
{}^{r}\Omega_{k}^{k(2m-2),2}=&\langle v_{2m-2}^{k-5}v_{2m}w_{4m-7}w_{4m-3}, v_{2m-4}v_{2m-2}^{k-6}v_{2m}w_{4m-5}w_{4m-3}, v_{2m-2}^{k-4}w_{4m-5}w_{4m-3}\rangle.
\end{align*}
The first sub-complex is exact because we have non-trivial differential:
$$\partial(v_{2m-2}^{k-5}v_{2m}w_{4m-5}w_{4m-3})=v_{2m-2}^{k-3}v_{2m}w_{4m-3}.$$ 
Now we investigate the second sub-complex. The differential $\partial$ is define on the bases of second sub-complex as:
\begin{align*}
\partial(v_{2m-2}^{k-1}v_{2m})=&0,\\
\partial(v_{2m-4}v_{2m-2}^{k-4}v_{2m}w_{4m-3})=&0,\\
\partial(v_{2m-2}^{k-3}v_{2m}w_{4m-5})=&v_{2m-2}^{k-1}v_{2m},\\
\partial(v_{2m-2}^{k-2}w_{4m-3})=&2v_{2m-2}^{k-1}v_{2m},\\
\partial(v_{2m-2}^{k-5}v_{2m}w_{4m-7}w_{4m-3})=&2v_{2m-4}v_{2m-2}^{k-4}v_{2m}w_{4m-3},\\
\partial(v_{2m-4}v_{2m-2}^{k-6}v_{2m}w_{4m-5})=&v_{2m-4}v_{2m-2}^{k-4}v_{2m}w_{4m-3},\\
\partial(w_{4m-3}, v_{2m-2}^{k-4}w_{4m-5}w_{4m-3})=&2v_{2m-4}v_{2m-2}^{k-4}v_{2m}w_{4m-3}+v_{2m-2}^{k-2}w_{4m-3}+\\
+&2v_{2m-2}^{k-3}v_{2m}w_{4m-5}. 
\end{align*}
Note that the monomial $v_{2m}^{a}$ is zero in ${}^{r}\Omega_{k}^{k(2m-2),2}(\text{CP}^{m})$ for $a>2.$ The dimensions of image and kernel of the map $\partial\,:\,{}^{r}\Omega_{k}^{k(2m-2)+1,1}(\text{CP}^{m})\longrightarrow {}^{r}\Omega_{k}^{k(2m-2)+2,0}(\text{CP}^{m})$ is 1 and 2, respectively. Moreover, the dimension of the image of map $\partial\,:\,{}^{r}\Omega_{k}^{k(2m-2),2}(\text{CP}^{m})\longrightarrow {}^{r}\Omega_{k}^{k(2m-2)+1,1}(\text{CP}^{m})$ is 2. From these computations, we conclude that the cohomology groups $$H^{k(2m-2)+i}(C_{k}(\text{CP}^{m}))$$ are vanish for $k\geq8$ and $i=1, 2, 3.$ We know that $\text{dim}H^{i}(M;\mathbb{Q})\cong \text{dim}H_{i}(M;\mathbb{Q}).$ Hence, for $i, m\in\mathbb{N},$ the Hilbert function 
$$k\mapsto \text{dim}H_{k(2m-2)+i}(C_{k}(\text{CP}^{m});\mathbb{Q})$$
is eventually a quasi-polynomial.  $\hfill \square$\\

\noindent\textbf{Acknowledgement}\textit{.} The Author gratefully acknowledge the support from the ASSMS GC, university Lahore. This research is partially supported by Higher Education Commission of Pakistan.

\vskip 0,65 true cm

\medskip

\null\hfill  Abdus Salam School of Mathematical Sciences,\\
\null\hfill  GC University Lahore, Pakistan. \\
\null\hfill E-mail: {yameen99khan@gmail.com}

\end{document}